\documentclass[11pt]{amsart}
\usepackage{amscd}
\usepackage{bbm}
\usepackage{dsfont}
\usepackage{enumerate}
\usepackage{latexsym}
\usepackage{srcltx}
\usepackage{amsfonts}
\usepackage{amssymb}
\usepackage{amsmath}
\usepackage{datetime}
\usepackage{setspace}
\usepackage{enumerate}
\usepackage{amsthm}
\usepackage{graphicx}
\usepackage{epstopdf}
\newtheorem{theorem}{Theorem}[section]
\newtheorem{proposition}[theorem]{Proposition}
\newtheorem{lemma}[theorem]{Lemma}
\newtheorem{corollary}[theorem]{Corollary}

\theoremstyle{definition}
\newtheorem{example}[theorem]{Example}

\newtheorem{remark} [theorem] {Remark}

\input epsf
\begin{document}

\title{ Spectrum of Weighted Composition Operators \\
Part IV \\
Spectrum and essential spectra of weighted composition operators
in spaces of smooth functions on $[0,1]$}

\author{A. K. Kitover}

\address{Community College of Philadelphia, 1700 Spring Garden St., Philadelphia, PA, USA}

\email{akitover@ccp.edu}

\subjclass[2010]{Primary 47B33; Secondary 47B48, 46B60}

\date{\today}

\keywords{Disjointness preserving operators, spectrum, Fredholm spectrum, essential spectra}
\begin{abstract} We provide a description of the spectrum and essential spectra of invertible weighted composition operators acting in spaces $C^{(n)}[0,1]$ and Sobolev spaces $ W^{n,X} = \{ f        \in C^{(n-1)}[0,1] : f^{(n)}        \in X \}$ where $X$ is an interpolation space between $L^1[0,1]$ and $L^\infty[0,1]$.

\end{abstract}
\maketitle

\markboth{A.K.Kitover}{Spectrum of weighted composition operators. V}

\section{Introduction}

In~\cite{AAK} and~\cite{Ki1} -~\cite{Ki3} the author studied the spectrum and essential spectra of disjointness preserving operators in Banach lattices. Here we apply the results obtained in the aforementioned papers to describe essential spectra of invertible weighted composition operators in some algebras of smooth functions on the interval $[0,1]$. To be more precise, in Section 2 we describe essential spectra of invertible weighted compositions acting on spaces $C^{(n)}[0,1]$ of $n$ times continuously differentiable functions on $[0,1]$, and in Section 3 we deal with essential spectra of invertible weighted compositions on Banach algebras $Lip_1^n[0,1]$ of functions with the $n^{th}$ derivative in the space $Lip_1$. Finally, in Section 4 we consider invertible weighted compositions on Sobolev spaces $W^{n,X}$ (see the precise definition at the beginning of Section 4) where $X$ is an interpolation space between $L^1(0,1)$ and $L^\infty(0,1)$.

\bigskip
We use the following standard notations.

\noindent $\mathds{N}$ - the set of all natural numbers.

\noindent $\mathds{Z}$ - the set of all integers.

\noindent $\mathds{R}$ the field of all real numbers.

\noindent $\mathds{C}$ - the field of all complex numbers.

\noindent We will denote the unit circle by $\Gamma$, i.e
\[
  \Gamma = \{\lambda \in \mathds{C} : |\lambda| = 1\}.
\]

All linear spaces are considered over $\mathds{C}$.

Let $X$ be a Banach space and $T: X    \rightarrow X$ be a bounded linear operator.

\noindent We denote the spectrum of $T$ in $X$ by $\sigma (T,X)$ and its spectral radius by $\rho (T,X)$.

\noindent We will write $\sigma (T)$ and $\rho (T)$ when it cannot cause any ambiguity.

\noindent We consider the following partition of $\sigma (T)$,
\[
  \sigma (T) = \sigma_{a.p.}(T)    \cup    \sigma_r(T),
\]
where \footnote{The definitions of $\sigma_{a.p}(T)$ and $\sigma_r(T)$ while not universally adopted are convenient for our purposes.} $\sigma_{a.p.}(T)$ is the approximate point spectrum of $T$:
\[
  \sigma_{a.p.}(T) = \{\lambda    \in \mathds{C} :    \exists    \{x_n : n    \in \mathds{N}\}    \subset X, \|x_n\|=1,
  Tx_n - \lambda x_n \mathop    \rightarrow \limits_{n \to \infty} 0\},
\]
and $\sigma_r(T)$ is the residual spectrum of $T$:
\[
  \sigma_r(T) = \{\lambda    \in \sigma (T) : (\lambda I - T)X \; \text{is closed in} \; X\}.
\]
 We adopt the following definitions of essential spectra from~\cite{EE}.

\[
  \sigma_1(T) = \{\lambda   \in \mathds{C} : \lambda I - T \; \text{is not semi-Fredholm}\}.
\]

\[
  \sigma_2(T) = \sigma_1(T)   \cup   \{\lambda   \in \sigma (T) : \lambda I - T \; \text{is semi-Fredholm and} \dim{ \ker{(\lambda I - T)}} = \infty\}.
\]

\noindent Equivalently, $\lambda   \in \sigma_2(T)$ if and only if there is a  sequence $x_n   \in X$ such that no its subsequence is convergent in norm (i.e. the sequence $x_n$ is \textit{singular}), $\|x_n\|=1$, and $Tx_n - \lambda x_n   \rightarrow 0$.

\[
  \sigma_3(T) = \{\lambda   \in \mathds{C} : \lambda I - T \; \text{is not Fredholm} \}.
\]
 \[
   \sigma_4(T) = \sigma_3(T)   \cup   \{\lambda   \in \mathds{C} : \lambda I - T \; \text{is Fredholm but} \; ind(\lambda I - T)   \neq   0 \}.
 \]

$\sigma_5(T) = \sigma (T) \setminus \{\zeta   \in \mathds{C} : $ there is a component $C$ of the set $\mathds{C} \setminus \sigma_1(T)$ such that $\zeta   \in C$ and the intersection of $C$ with the resolvent set of $T$ is not empty $\}$.

It is well known and easy to see that
\[
  \sigma_1(T)  \subseteq \sigma_2(T)  \subseteq \sigma_3(T)  \subseteq \sigma_4(T)  \subseteq \sigma_5(T)
\]
and all the inclusions can be proper. But the essential spectral radius is the same no matter which of the five sets above we consider~\cite{EE}. We will denote it by $\rho_e(T,X)$ or just by $\rho_e(T)$.

We will also consider the set $\sigma_2(T^\prime)$, where $T^\prime$ is the Banach adjoint of $T$. It is immediate that
\[
  \sigma_2(T^\prime)= \sigma_1(T)   \cup   \{\lambda   \in \sigma (T) : \lambda I - T \; \text{is semi-Fredholm and}\; def(\lambda I - T) = \infty\}.
\]
Therefore $\sigma_1(T) =\sigma_2(T)  \cap \sigma_2(T^\prime)$ and $\sigma_3(T) =\sigma_2(T)  \cup  \sigma_2(T^\prime)$.

 If a function $f  \in C[0,1]$ is $n$ times differentiable we will denote its derivative of order $i$, $i \leq n$, by $f^{(i)}$. In these notations $f^{(0)}$ means $f$.

 We will keep notation $f^{(n)}$ if $f$ is continuously differentiable $n-1$ times and has a derivative of order $n$ from $L^1(0,1)$.

 Let $\varphi$ be a homeomorphism of $[0,1]$ onto itself. By $\varphi^n, n  \in \mathds{Z}$ we denote the $n^{th}$ iteration of $\varphi$. More precisely, $\varphi^0 = \varphi$ and $\varphi^{n+1} = \varphi^n  \circ \varphi, n  \in \mathds{Z}$.

 Let $w  \in C[0,1]$. Then we define $w_n$ as
 \[
   w_n = w(w  \circ \varphi ) \ldots (w  \circ \varphi^{n-1}), n  \in \mathds{N}.
 \]

 We will need the following corollary that follows from~\cite[Theorem 3.2  and Corollary 3.3]{Ki3}.

 \begin{corollary} \label{c1.10} Let $\varphi$ be a homeomorphism of $[0,1]$ onto itself and let $w  \in C[0,1]$. Consider the weighted composition operator $T$ on $C[0,1]$
 \[
   (Tf)(x) = w(x)f(\varphi (x)), f  \in C[0,1], x  \in [0,1].
 \]
 Then
 \begin{enumerate}
   \item $\sigma_2(T) = \sigma_{a.p.}(T)$.
   \item $\sigma_2(T^\prime) = \sigma_{a.p.}(T^\prime)$.
   \item $\sigma_3(T) = \sigma (T)$.
 \end{enumerate}
 \end{corollary}

\section{Spectrum and essential spectra of invertible weighted composition operators on $C^{(n)}[0,1]$.}

\noindent In this section we will consider invertible weighted composition operators on the spaces $C^{(n)}[0,1]$, $n        \in \mathds{N}$, of $n$ times continuously differentiable functions on $[0,1]$ endowed with the norm
\[ \|f\| = \sum \limits_{i=0}^{n-1} |f^{(i)}(0)| + \|f^{(n)}\|_\infty \]
where $f^{(0)}$ means $f$.

One of the main results of this section is the following theorem.

\begin{theorem} \label{t1.1}
Let $\varphi$ be a diffeomorphism of $[0,1]$ onto itself such that $\varphi (0)=0$ and let $w$ be an invertible element of the algebra $ C^{(1)}[0,1]$. Let $F, F        \subseteq [0,1],$ be the set of all fixed points of $\varphi$ and let $I$ be the subset of $F$consisting of points isolated in $F$.  Let $T$ be the weighted composition operator on $C^{(1)}[0,1]$ defined as
$$ (Tf)(x) = w(x)f(\varphi (x), \; f        \in C^{(1)}[0,1], \; x        \in [0,1]. $$
Let $A$ be the weighted composition operator on $C[0,1]$ defined as
$$ (Ag)(x) = w(x)\varphi^\prime(x)g(\varphi (x)), \; g        \in C[0,1], \; x        \in [0,1].$$
Then
$(I)$ $$\sigma_i(T) = \sigma_i(A), \; i=1, \ldots 5.$$

$(II)$ $$\sigma_i(T) = \sigma (A), i=3,4,5.$$

$(III)$ $\sigma (A)        \subseteq \sigma (T)$ and
$$ \sigma (T) \setminus \sigma (A) = w(I) \setminus \sigma (A).$$
Moreover, if $\lambda        \in \sigma (T) \setminus \sigma (A)$ then $\lambda$ is an isolated eigenvalue of $T$ of finite multiplicity.
\end{theorem}

\begin{proof} In $C^{(1)}[0,1]$ consider the closed subspace $C_0^{(1)}= \{f        \in C^{(1)}[0,1] \; : \; f(0) = 0\}$. Clearly $C_0^{(1)}$ is $T$-invariant. The space $C_0^{(1)}$ is isometrically isomorphic to $C[0,1]$ and it is immediate to see that the restriction $T_0$ of $T$ onto $C_0^{(1)}$ is similar to the operator $B$ on $C[0,1]$ defined as
\[
  (Bf)(x) = w^\prime(x) \int \limits_0^{\varphi (x)} f(t)dt + (Af)(x), f        \in C[0,1], x        \in [0,1]. \eqno(1.1)
\]

It follows from $(1.1)$ that the operator $B - A$ is compact on $C[0,1]$ and because the spectra $\sigma_i, i=1,2,3,4$ are invariant under compact perturbations we have
\[
  \sigma_i(T_0) = \sigma_i(B) = \sigma_i(A), i= 1,2,3,4.
\]
Next notice that because $C_0^{(1)}$ has codimension one in $C^{(1)}[0,1]$ we have
\[
  \sigma_i(T) = \sigma_i(A), i= 1,2,3,4.
\]
It follows from Corollary~\ref{c1.10} that $\sigma_3(A)=\sigma_4(A)=\sigma_5(A) =\sigma (A)$. Therefore $\sigma (A)        \subseteq \sigma (T)$. We have to postpone the proof of the remaining part of statements $(I)$ and $(II)$ : $\sigma_5(T) = \sigma (A)$ until we finish the proof of $(III)$.

Assume that $\lambda        \in \sigma (T) \setminus (w(F)        \cup        \sigma (A))$. We have to consider two possibilities.

\noindent 1. $\lambda        \in \sigma_{a.p.}(T)$. Let $f_n        \in C^{(1)}[0,1]$, $\|f_n\|=1$, and $Tf_n -\lambda f_n        \rightarrow 0$. If $\|f_n\|_\infty        \rightarrow 0$ then $\|f_n^{(1)}\|_\infty        \rightarrow 1$ and $Af_n^{(1)} - \lambda f_n^{(1)}        \rightarrow 0$ whence $\lambda        \in \sigma (A)$.

If on the other hand $\|f_n\|_\infty \not        \rightarrow 0$ then, because the set $\{f_n\}$ is compact in $C[0,1]$, we can assume without loss of generality that $f_n$ converge to some $f$ in $C[0,1]$. Moreover,
 $$|f(u) - f(v)| \leq |u-v|, u,v        \in [0,1] \eqno{(1.2)}$$
  and $w(x)f(\varphi (x)) = \lambda f(x), x        \in [0,1]$. Our assumption that $\lambda        \in \sigma (T) \setminus w(F)$ implies that the set $[0,1] \setminus F$ is not empty and that there is an interval $(a,b)        \subset [0,1]$ such that $a,b        \in F$, $(a,b)        \cap F = \emptyset$, and $f \not        \equiv        0$ on $(a,b)$. For definiteness we will assume now that $\varphi (x) <x$ on $(a,b)$. The case when $\varphi (x) > x$ on $(a,b)$ can be considered similarly.

\noindent Let $R_1 = |w(a)|\varphi^\prime(a)|$ and $R_2 = |w(b)|\varphi^\prime(b)|$.  Then (see~\cite[Theorem 3.2]{Ki3}) $\sigma (A, C[a,b])$ is an annulus with the radii $R_1$ and $R_2$.  Notice that because $\lambda \not        \in \{w(a), w(b)\}$ we have $f(a) = f(b) = 0$. Next, because $\lambda \not        \in \sigma (A)$ and $\sigma (A,C[a,b])        \subseteq \sigma (A)$ we have that either $|\lambda| < min(R_1, R_2)$ or $|\lambda| > max(R_1, R_2)$.

 Assume first that $|\lambda| > max(R_1, R_2)$. Fix $c        \in (a,b)$ such that $f(c)        \neq        0$. Then
$$ w_n(c)f(\varphi^n(c)) = \lambda^nf(c), n        \in \mathds{N}.$$
Obviously $\lim \limits_n |\lambda^n f(c)|^{1/n} = |\lambda|$. On the other hand applying (1.2) and the mean value theorem we see that

$$\limsup \limits_n |w_n(c)f(\varphi^n(c))|^{1/n} = \lim \limits_n |w_n(c)|^{\frac{1}{n}} \limsup \limits_n |f(\varphi^n(c))|^{1/n} =$$
$$= |w(a)| \limsup \limits_n |f(\varphi^n(c))|^{1/n} \leq |w(a)|\limsup \limits_n \{(\varphi^n)^\prime (c_n)|\}^{\frac{1}{n}}, $$
where $a < c_n < \varphi^n(c)$. Therefore
$$ \limsup \limits_n |w_n(c) f(\varphi^n(c))|^{1/n} \leq |w(a)\varphi^\prime(a)| = R_1 < |\lambda|,$$
  a contradiction.

 Next let us assume that $|\lambda| < min(R_1, R_2)$. Then $1/|\lambda| > max(1/R_1, 1/R_2)$ Again let $c        \in (0,1)$ and $f(c)        \neq        0$. Then
$$ w_n(\varphi^{-n}(c))f(\varphi^{-n}(c)) = \lambda^{-n}f(c)$$
and we come to a contradiction in a similar way.

\noindent 2. $\lambda        \in \sigma_r(T)$. In this case the operator $\lambda I - T$ is semi-Fredholm and its index is different from 0. Then the same is true for operator $B$. But the index of a semi-Fredholm operator is invariant under compact perturbations in contradiction with our assumption that the operator $\lambda I - A$ is invertible.

We have proved that if $\lambda        \in \sigma (T) \setminus \sigma (A)$ then $\lambda        \in w(F)$. Let $a        \in F$ be such that $\lambda = w(a)$. If $a$ is not an isolated point in $F$ then $\varphi^\prime(a) = 1$ and clearly $w(a)        \in \sigma (A)$. Thus $a$ must be an isolated point in $F$ and moreover $\lambda$ must be an isolated point in $\sigma (T) \setminus \sigma (A)$ and therefore it is an isolated point in $\sigma (B) \setminus \sigma (A)$. Recalling again that $B-A$ is a compact operator on $C[0,1]$ we see that $\lambda$ is an isolated eigenvalue of $T$ of finite multiplicity. That proves $(III)$. Now the statement $\sigma_5(T) = \sigma_5(A) = \sigma (A)$ follows immediately from $(III)$ and the definition of the set $\sigma_5(T)$.

\end{proof}

We will now state some corollaries of Theorem~\ref{t1.1}. In particular we will provide a precise description of the sets $\sigma (T)$ and $\sigma_i(T), i= 1, \ldots ,5$.  In the set $F$ of fixed points of $\varphi$ we will consider already mentioned subset $I$ of isolated points and subset $E= Int_{[0,1]} F$. In corollaries below we assume that $T$ is an invertible weighted composition operator on $C^{(1)}[0,1]$ and that $\varphi (0)=0$.

\begin{corollary} \label{c1.1}
\begin{enumerate}
  \item   \[
          \sigma (T) = w(I)        \cup        w(E)        \cup        \sigma
        \]
  where $\sigma$ is a compact rotation invariant subset of $\mathds{C}$.
  \item \[
          \sigma_i(T) = w(E)        \cup        \sigma^i, i=1, \ldots, 5
        \]
  where $\sigma^i$ are compact rotation invariant subsets of $\mathds{C}$.
  \item In particular, if $F$ is nowhere dense in $[0,1]$ then the sets $\sigma_i(T), i= 1, \ldots, 5 $ are rotation invariant. Moreover, in this case the sets $\sigma_i(T), i=3,4,5,$ coincide with the annulus (or circle) with the radii $R_1$ and $R_2$ where
      $$
        R_1 = \rho_e(T) = \rho (A) = \max \limits_{\mu        \in \mathcal{M}} \exp{\int{\ln{(|w|\varphi^\prime)}}} d\mu = $$
       $$ = \max \limits_{[0,1]} |w(x)|\varphi^\prime (x)= \max(\max \limits_{x        \in F \setminus I} |w(x)|, \max \limits_{x        \in I} |w(x)|\varphi^\prime(x))$$

      and

       $$ R_2 = \frac{1}{\rho_e(T^{-1})} =\frac{1}{ \rho (A^{-1})} = \min \limits_{\mu        \in \mathcal{M}}, \exp{\int{\ln{(|w|\varphi^\prime)}}} d\mu = $$
       $$ = \min \limits_{[0,1]} |w(x)|\varphi^\prime (x)= \min(\min \limits_{x        \in F \setminus I} |w(x)|, \min \limits_{x        \in I} |w(x)|\varphi^\prime(x))$$

   where $\mathcal{M}$ is the set of all probability regular $\varphi$-invariant Borel measures on $F$.

\end{enumerate}

\end{corollary}

\begin{corollary} \label{c1.2} The following conditions are equivalent.
\begin{enumerate}
\item
\[
  \lambda        \in \sigma (T) \setminus \sigma (A).
\]

\item The set $\{x        \in F : w(x) =\lambda \}$ consists of a finite number of points $a_1, \ldots, a_n        \in I$ and either $|\lambda| > \rho (A)$ or $|\lambda| < 1/\rho (A^{-1})$.

    \noindent Moreover, the multiplicity of the eigenvalue $\lambda$ is $n$.
\end{enumerate}
\end{corollary}

\begin{corollary} \label{c1.3} The following conditions are equivalent.
\begin{enumerate}
  \item $\lambda        \in \sigma_2(T)$.
  \item At least one of the listed below conditions is satisfied.
  \begin{enumerate}[(a)]
    \item $       \exists        \; a        \in Int_{[0,1]}F$ such that $\lambda = w(a)\varphi^\prime(a)$.
    \item $       \exists        \; a        \in F \setminus Int_{[0,1]}F$ such that $|\lambda| = |w(a)|\varphi^\prime(a)$.
    \item $       \exists        a,b        \in F$ such that $a < b$, $\varphi (x) < x, x        \in (a,b)$, and
    \[
      |w(b)|\varphi^\prime(b) < |\lambda| < |w(a)|\varphi^\prime(a)|.
    \]
    \item $       \exists        a,b        \in F$ such that $a < b$, $\varphi (x) > x, x        \in (a,b)$, and
    \[
      |w(a)|\varphi^\prime(a) < |\lambda| < |w(b)|\varphi^\prime(b)|.
    \]
  \end{enumerate}
\end{enumerate}

\end{corollary}

\begin{corollary} \label{c1.4} The following conditions are equivalent.
\begin{enumerate}
  \item $\lambda        \in \sigma_2(T^\prime)$.
  \item At least one of the listed below conditions is satisfied.
  \begin{enumerate}[(a)]
    \item $       \exists        \; a        \in Int_{[0,1]}F$ such that $\lambda = w(a)\varphi^\prime(a)$.
    \item $       \exists        \; a        \in F \setminus Int_{[0,1]}F$ such that $|\lambda| = |w(a)|\varphi^\prime(a)$.
    \item $       \exists        a,b        \in F$ such that $a < b$, $\varphi (x) < x, x        \in (a,b)$, and
    \[
      |w(a)|\varphi^\prime(a) < |\lambda| < |w(b)|\varphi^\prime(b)|.
    \]
    \item $       \exists        a,b        \in F$ such that $a < b$, $\varphi (x) > x, x        \in (a,b)$, and
    \[
      |w(b)|\varphi^\prime(b) < |\lambda| < |w(a)|\varphi^\prime(a)|.
    \]
  \end{enumerate}
\end{enumerate}

\end{corollary}

Finally notice that a description of $\sigma_1(T)$ follows from Corollaries~\ref{c1.3} and~\ref{c1.4} and the fact that $\sigma_1(T) = \sigma_2(T)        \cap \sigma_2(T^\prime)$.

To illustrate the above statements let us consider a simple example.

\begin{example} \label{e1.1} Let $\varphi$ be a diffeomorphism of $[0,1]$ such that $\varphi (x) < x, x        \in (0,1)$. Notice that in this case $|\varphi^\prime(0)| \leq 1$ and $|\varphi^\prime(1)| \geq 1$. Let $w        \in C^{(1)}[0,1]$. There are the following possibilities.
\begin{enumerate}
  \item $|w(0)|\varphi^\prime(0) =|w(1)|\varphi^\prime(1)= \rho$. Then
   $$\sigma_i(T) = \rho \Gamma, i= 1 \ldots , 5.$$
  \item $|w(0)|\varphi^\prime(0)= \rho_1 > |w(1)|\varphi^\prime(1)= \rho_2$. Let $A$ be the closed annulus with the radii $\rho_1$ and $\rho_2$. Then
  \[
    \sigma_i(T) = A, i=2, \ldots , 5
  \]
  and
  \[
    \sigma_1(T) = \sigma_2(T^\prime) = \rho_1 \Gamma       \cup       \rho_2 \Gamma .
  \]
   \item $|w(0)|\varphi^\prime(0)= \rho_1 < |w(1)|\varphi^\prime(1)= \rho_2$. Let $A$ be the closed annulus with the radii $\rho_1$ and $\rho_2$. Then
  \[
   \sigma_2(T^\prime)= \sigma_i(T) = A, i=3, \ldots , 5
  \]
  and
  \[
    \sigma_1(T) = \sigma_2(T) = \rho_1 \Gamma       \cup       \rho_2 \Gamma .
  \]
\end{enumerate}
Moreover, if $\lambda       \in \sigma (T)$ and $|\lambda| > \rho (A)$ then $\lambda = w(0)$, while if $|\lambda| < 1/\rho (A^{-1})$ then $\lambda = w(1)$. In both cases the corresponding eigenspace is one-dimensional.
\end{example}

Let us briefly address the case when $\varphi$ is a diffeomorphism of $[0,1]$ onto itself such that $\varphi (0)=1$. In this case the operator $T^2$ is of the type considered in Theorem~\ref{t1.1} and a description of the spectrum and essential spectra of $T$ can be obtained from that theorem and some rather simple additional considerations. Notice that in this case $\varphi$ has exactly one fixed point in $(0,1)$ and other periodic points are of period 2. By the reasons just outlined we omit the proof of the following theorem.

\begin{theorem} \label{t1.2} Let $\varphi$ be a diffeomorphism of $[0,1]$ onto itself such that $\varphi (0)=1$. Let $w$ be an invertible element of the algebra $C^{(1)}[0,1]$
 and $T$ be the weighted composition operator
\[
  (Tf)(x) = w(x)f(\varphi (x)), \; f       \in C^{(1)}[0,1], \; x       \in [0,1].
\]

Let $A$ be the weighted composition operator on $C[0,1]$ defined as
\[
  (Ag)(x) =w(x)\varphi^\prime(x)g(\varphi (x)), g       \in C[0,1], x       \in [0,1].
\]
Let $\Pi$ be the set of all $\varphi$-periodic points and $I$ be the subset of $\Pi$ consisting of all points isolated in $\Pi$
Then
\begin{enumerate}
  \item $\sigma_i(T) = \sigma_i(A), i=1, \ldots, 5$. Moreover, if the set $\Pi$ is nowhere dense in $[0,1]$ then the sets $\sigma_i(T), i = 1, \ldots, 5$ are rotation invariant.
  \item If $\lambda       \in \sigma (T) \setminus \sigma (A) $ then there is an $a      \in I$such that $\lambda^2 = w(a)w(\varphi (a))$. Moreover, $\lambda$ is an isolated eigenvalue of $\sigma (T)$ of multiplicity 1.
\end{enumerate}

\end{theorem}

At the end of this section we will consider invertible weighted automorphisms of the Banach algebra $C^n(0,1)$, $n        \in \mathds{N}$ of all functions on $[0,1]$ continuously differentiable $n$ times. To get rid of nonessential details we will assume that $\varphi (0) = 0$ and that the set $F$ of all $\varphi$-fixed points is nowhere dense in $[0,1]$. The proof of the next theorem goes along the same lines as the proof of Theorem~\ref{t1.1} and is omitted.

\begin{theorem} \label{t1.3} Let $n        \in \mathds{N}$, $\varphi$ be a homeomorphism of $[0,1]$ onto itself such that
\begin{itemize}
  \item $\varphi , \varphi^{-1}        \in C^{(n)}[0,1]$.
  \item $\varphi (0) = 0$.
  \item The set $F$ of $\varphi$-fixed points is nowhere dense in $[0,1]$.
 \end{itemize}
 Let $w$ be an invertible element of the algebra $C^{(n)}[0,1]$ and $T$ be the weighted composition operator on $C^{(n)}[0,1]$ defined as
 $$Tf = w(f        \circ \varphi ), f        \in C^{(n)}[0,1].$$
Then
\begin{enumerate}
  \item $\sigma_3(T) = \sigma_4(T) = \sigma_5(T) = \sigma (A)$, where $A$ is the weighted composition operator on $C[0,1]$ defined as
  $$Ag = w \varphi^{(n)} (g        \circ \varphi ), g        \in C[0,1].$$
  \item $\sigma (A)$ is the annulus with the radii $R = \max \limits_{x        \in F} |w(x)\varphi^{(n)}(x)|$ and $r = \min \limits_{x        \in F} |w(x)\varphi^{(n)}(x)|$.
  \item $\lambda        \in \sigma (T) \setminus \sigma (A)$ if and only if there is an $a        \in [0,1]$ such that $a$ is an isolated point in $F$, $\lambda = w(a)$, and either $|\lambda| < r$ or $|\lambda| > R$.
  \item If $\lambda        \in \sigma (T) \setminus \sigma (A)$ then $\lambda$ is an isolated eigenvalue of $T$ and the dimension of the corresponding eigenspace is $n$.
  \item The sets $\sigma_i(T), i=1, \ldots , 5$, are rotation invariant.
  \item $\lambda        \in \sigma_1(T)$ if and only if either

  \noindent there is an $x        \in F$ such that $|\lambda| = |w(x)|$ or there are two disjoint intervals $(a,b)$ and $(c,d)$ complementary to $F$ and such that one of the following conditions holds.

  \noindent (a) $\varphi (x) < x$ on $(a,b)        \cup        (c,d)$, $|w(a)\varphi^{(n)}(a)| < |\lambda| < |w(b)\varphi^{(n)}(b)| $ and
  $|w(c)\varphi^{(n)}(c)| > |\lambda| > |w(d)\varphi^{(n)}(d)| $.

  \noindent (b) $\varphi (x) > x$ on $(a,b)        \cup        (c,d)$, $|w(a)\varphi^{(n)}(a)| < |\lambda| < |w(b)\varphi^{(n)}(b)| $ and
  $|w(c)\varphi^{(n)}(c)| > |\lambda| > |w(d)\varphi^{(n)}(d)| $.

  \noindent (c) $\varphi (x) > x$ on $(a,b)$, $\varphi (x) < x$ on $(c,d)$, $|w(a)\varphi^{(n)}(a)| < |\lambda| < |w(b)\varphi^{(n)}(b)| $ and
  $|w(c)\varphi^{(n)}(c)| < |\lambda| < |w(d)\varphi^{(n)}(d)| $.

  \noindent (d) $\varphi (x) > x$ on $(a,b)$, $\varphi (x) < x$ on $(c,d)$, $|w(a)\varphi^{(n)}(a)| > |\lambda| > |w(b)\varphi^{(n)}(b)| $ and
  $|w(c)\varphi^{(n)}(c)| > |\lambda| > |w(d)\varphi^{(n)}(d)| $.

  \item $\lambda        \in \sigma_2(T) \setminus \sigma_1(T)$ if and only if $\lambda \not        \in \sigma_1(T)$ and there is an interval $(a,b)$ complementary to $F$ such that either $\varphi (x) < x $ on $(a,b)$ and $|w(a)\varphi^{(n)}(a)| < |\lambda| < |w(b)\varphi^{(n)}(b)| $ or $\varphi (x) > x $ on $(a,b)$ and $|w(a)\varphi^{(n)}(a)| > |\lambda| > |w(b)\varphi^{(n)}(b)| $. Moreover, in this case $nul(\lambda I - T) = 0$.

  \item $\lambda        \in \sigma_2(T^\star) \setminus \sigma_1(T)$ if and only if $\lambda \not        \in \sigma_1(T)$ and there is an interval $(a,b)$ complementary to $F$ such that either $\varphi (x) < x $ on $(a,b)$ and $|w(a)\varphi^{(n)}(a)| > |\lambda| > |w(b)\varphi^{(n)}(b)| $ or $\varphi (x) > x $ on $(a,b)$ and $|w(a)\varphi^{(n)}(a)| < |\lambda| < |w(b)\varphi^{(n)}(b)| $. Moreover, in this case $def(\lambda I - T) = 0$.
\end{enumerate}
\end{theorem}

\section{ Essential spectra of invertible weighted automorphisms of the algebras $Lip_1^n[0,1]$.}

In this section we will consider Banach algebras $Lip_1^n[0,1]$, $       \in \mathds{N}$ of functions continuously differentiable $n-1$ times on $[0,1]$ and with the $n^{th}$ derivative in $L^\infty(0,1)$, i.e. there is $g       \in L^\infty(0,1)$ such that $f^{(n-1)}(x) = \int \limits_0^x f(t)dt, x      \in [0,1]$. It is well known that the space $Lip_1^n[0,1]$ endowed with the norm
\[
  \|f\| = \sum \limits_{i=0}^{n-1} |f^{(i)}(0)| + \|f^{(n)}_\infty\|
\]
is a Banach algebra.

Our nearest goal is to obtain a description of essential spectra of invertible weighted composition operators acting on the space $Lip_1[0,1]$. Thus we consider operators on this space of the form
$$Tf(x) = w(x)f(\varphi (x)), f        \in Lip_1(0,1), x        \in [0,1],$$
where $w$ is an invertible element of the algebra $Lip_1[0,1]$ and $\varphi$ is a homeomorphism of $[0,1]$ onto itself such that $\varphi$ and $\varphi^{-1}$ belong to $Lip_1[0,1]$.

 Temporarily we will assume that $\varphi (x) < x$ on $(0,1)$. Let $X$ be the hyperplane in $Lip_1[0,1]$ defined as $X = \{f        \in Lip_1(0,1): \, f(0) = 0\}$. Then $TX=X$ and the restriction $T|X$ is similar to the operator $B$ on $L^\infty(0,1)$ defined as
$$ (Bf)(x) = w^\prime(x) \int \limits_0^{\varphi (x)} f(t)dt + w(x)\varphi^\prime(x)f(\varphi (x)), f        \in L^\infty(0,1), x        \in [0,1].)$$
Let $A$ be the weighted composition operator on $L^\infty(0,1)$ defined as
$$Af(x) = w(x)\varphi^\prime(x)f(\varphi (x)), f        \in L^\infty(0,1), x        \in [0,1].$$
Then obviously the operator $B - A$ is a compact operator on $L^\infty(0,1)$ and taking into consideration that $dim(Lip_1(0,1)/X = 1$ we get
$$\sigma_i(T) =\sigma_i(A), i=1,2,3,4.$$
In the next few lemmas we will look at some properties of the spectrum and essential spectra of operator $A$. It is useful to keep in mind that $A$ can be identified with an invertible weighted composition operator $\hat{A}$ on $C(Q)$ where $Q$ is the Gelfand compact of the algebra $L^\infty(0,1)$,
$$\hat{A} \hat{f}(t)=\hat{w}(t)\widehat{\varphi^\prime(t)}\hat{f}(\psi (t)), t        \in Q.$$
In the above formula $\hat{f}$ means the Gelfand transform of $f$, $f        \in L^\infty(0,1)$ and $\psi$ is the homeomorphism of $Q$ that corresponds to the isomorphism $J$ of $L^\infty(0,1)$ defined as $Jf(x) = f(\varphi (x))$ (The conditions we put on $\varphi$ guarantee that $J$ is indeed an isomorphism).

\begin{lemma} \label{l1} The set $\sigma (A) = \sigma (\hat{A})$ is rotation invariant.
\end{lemma}

\begin{proof} The set of $\psi$-periodic points is nowhere dense in $Q$ and actually by Frolik's theorem (see~\cite[Theorem 6.25, page 150]{Wa}) it is empty. The statement follows now from e.g.~\cite[Theorem 3.7]{Ki1}.
\end{proof}

\begin{lemma} \label{l2} The set $\sigma (A) = \sigma (\hat{A})$ is an annulus or a circle.
\end{lemma}

\begin{proof}  By Lemma~\ref{l1} it is enough to prove that the set $\sigma (\hat{A})$ is connected. Assume to the contrary that there is a positive number $r$ such that $\sigma (\hat{A}) = \sigma_1        \cup        \sigma_2$ where $\sigma_i       \neq       \emptyset, i=1,2$ and
$\sigma_1        \subset \{\lambda        \in \mathds{C} : |\lambda| < r\}$ and $\sigma_2        \subset \{\lambda        \in \mathds{C} : |\lambda| > r\}$. Without loss of generality we can assume that $r=1$. Then (see~\cite[Theorem 3.10]{Ki1})) there are two $\varphi$-invariant subsets $E$ and $F$ of $[0,1]$ both of positive Lebesgue measure, a positive integer $N$ and positive numbers $s$,$s<1$ and $t$, $t>1$ such that $mes(E) + mes(F) = 1$,
$$|w_N|(\varphi^N)^\prime < s^N \; \mathrm{a.e. on} \; E,$$
and
$$ |w_N|(\varphi^N)^\prime > t^N \; \mathrm{a.e. on} \; F.$$
Without loss of generality we can assume that $N=1$ and thus
\[
  |w|\varphi^\prime < s \; \text{a.e. on}\; E \; \text{and} \; |w|\varphi^\prime > t \; \text{a.e. on} \; F.
\]

 Because the sets $E$ and $F$ are $\varphi$-invariant and $w$ is continuous on $[0,1]$ we see that there are positive numbers $\varepsilon$ and $u$, such that $0 < \varepsilon, u <1$ and
 \[
   \|\varphi^\prime \chi_D\|_\infty < u \frac{1}{\|w^{-1}\chi_H\|_\infty} \eqno{(2.1)}
 \]
 where $D = E       \cap ((0, \varepsilon )       \cup       (1-\varepsilon, 1))$ and $H = F       \cap ((0, \varepsilon )       \cup       (1-\varepsilon, 1))$. Combining (2.1) with the fact that $\varphi (x) < x$ on $(0,1)$ it is not difficult to see that there are $m       \in \mathds{N}$ and $v       \in (0,1)$ such that
 \[
   \|\varphi_m^\prime \chi_E \|_\infty < v \frac{1}{\|(\varphi^\prime)_m^{-1} \chi_F\|}  \eqno{(2.2)}
 \]
 Consider on $L^\infty (0,1)$ the weighted composition operator defined as
 \[
   Sf = \varphi^\prime f      \circ \varphi, f       \in L^\infty (0,1).
 \]
 It follows from (2.2) that $\sigma (S)$ is the union of two disjoint nonempty rotation invariant subsets $\sigma (s, L^\infty (E))$ and $\sigma (S, L^\infty (F))$. Considering, if necessary, $S^{-1}$ we can assume that $\rho (S, L^\infty(E))<1$. Let $L_1$ be the subspace of $Lip_1(0,1)$ defined as follows
 \[
   L_1 = \{f       \in Lip_1(0,1) : f^\prime =0 \; \text{a.e. on} \; F\}.
 \]
 Then $L_1$ is a Banach subalgebra of $Lip_1(0,1)$ and the operator $\mathbf{S}$,
 \[
   \mathbf{S}f = f       \circ \varphi, f       \in L_1,
 \]
 is an automorphism of $L_1$. Because $\rho_{ess}(\mathbf{S}) = \rho (S,L^\infty (E)) < 1$ the intersection
 $\sigma (\mathbf{S}       \cap \Gamma$ consists of at most finite number of points. But by the Kamowitz - Scheinberg theorem~\cite{KS} the spectrum of an automorphism of a semisimple commutative Banach algebra must contain the unit circle, a contradiction.

\end{proof}

\begin{corollary} \label{c1} $\sigma_5(T) = \sigma_5(A) = \sigma (A)$.

\end{corollary}

\begin{lemma} \label{l3} $\sigma (T) \setminus \sigma (A)        \subseteq \{w(0), w(1)\}$.

\end{lemma}

\begin{proof} Let $\lambda        \in \sigma (T) \setminus \sigma (A)$. By Lemmas~\ref{l1} and ~\ref{l2} either $|\lambda| > \rho (A)$ or $\frac{1}{\lambda } > \rho (A^{-1})$. Assume first that $|\lambda| > \rho (A)$ and $\lambda        \neq        w(0)$. $\lambda$ is an isolated eigenvalue of $T$. Let $f$ be a corresponding eigenvector, then $f(0) =0$. Let $a        \in (0,1)$ be such that $f(a)        \neq        0$. Then $w_n(a)f(\varphi^n(a))=\lambda^n f(a), n        \in \mathds{N}$. Therefore
$$|\lambda| \leq |w(0)| \limsup \limits_n |f(\varphi^n(a))|^{\frac{1}{n}} \leq |w(0)|\limsup \limits_n(\|(f        \circ \varphi^n)^\prime|_{[0,a]} \|_\infty)^{\frac{1}{n}}.$$
Consider the closed subspace $L_a$ of $L^\infty(0,1)$ defined as
$$L_a = \{g        \in L^\infty(0,1) : supp \; g        \subseteq [0,a]\}.$$
Let $Cg = w(0)(g       \circ \varphi ), g        \in L^\infty(0,1)$. Then $L_a$ is invariant for $A$ and $C$ and (see e.g.~\cite[Theorem 3.23]{Ki1}) $\rho (A|L_a) = \rho (C|L_a)$. From here we get
$$|w(0)|\limsup \limits_n(\|(f        \circ \varphi^n)^\prime)|_{[0,a]} \|_\infty)^{\frac{1}{n}} \leq \rho (C|L_a) = \rho (A|L_a) \leq \rho (A) < |\lambda|,$$
a contradiction.

Assume now that $\lambda        \in \sigma (T) \setminus \sigma (A)$, $1/|\lambda| > 1/\rho (A^{-1})$, and $\lambda        \neq        w(1)$. We can bring these assumptions to a contradiction by considering $T^{-1}$ instead of $T$ and applying the same kind of reasoning as in the previous case.
\end{proof}

\begin{remark} \label{r2} The proof of Lemma~\ref{l3} shows that if $\lambda        \in \{w(0), w(1)\} \setminus \sigma (A)$ then the dimension of the corresponding eigenspace is 1.

\end{remark}

\begin{corollary} \label{c2} The spectra $\sigma_1(T)        \subseteq \sigma_2(T)        \subseteq \sigma_3(T)        \subseteq \sigma_4(T)        \subseteq \sigma_5(T)$, as well as
$\sigma_2(T^\prime)$ are rotation invariant.
\end{corollary}

\begin{proof} Follows from~\cite[Theorem 3.2]{Ki3}

\end{proof}

\begin{corollary} \label{c3} Let $a        \in (0,1)$, $L_a = \{f        \in L^\infty(0,1) : supp f        \subseteq [0,a]\}$, and $R_a = \{f        \in L^\infty(0,1) : supp f        \subseteq [a,1] \}$.

\noindent (1) Assume that $\lambda        \in \mathds{C}$ and $\rho (A|L_a) < |\lambda| < 1/\rho (A^{-1}|R_a)$. Then the operator $\lambda I - T$ has the left inverse; in particular, $\lambda        \in \sigma_2(T) \setminus \sigma_1(T)$.

\noindent (2) Assume that $\lambda        \in \mathds{C}$ and $\rho (A|L_a) > |\lambda| > 1/\rho (A^{-1}|R_a)$. Then the operator $\lambda I - T$ has the right inverse; in particular, $(\lambda I -T) Lip_1(0,1) = Lip_1(0,1)$ and $\lambda        \in \sigma_2(T^\prime) \setminus \sigma_1(T)$.

\end{corollary}

We will now refine the result stated in Corollary~\ref{c3}. We start with considering the case $w        \equiv        1$ when $T$ is an automorphism of $Lip_1(0,1)$.

\begin{lemma} \label{l4} Let $\varphi$ be a homeomorphism of $[0,1]$ onto itself such that $\varphi$ is an invertible element of the algebra $Lip_1(0,1)$ and $\varphi (x) <  x$ on $(0,1)$. Let $Tf = f       \circ \varphi, f        \in Lip_1(0,1)$. Then

\noindent (1) $\sigma_1(T) = \sigma_2(T)$

\noindent (2)  Let $0 < |\lambda| < 1$. Then $ \lambda        \in \sigma_2(T^\prime) \setminus \sigma_1(T)$ if and only if there are Lebesgue measurable disjoint subsets $E_1$ and $E_2$ of $[0,1]$, $\varepsilon > 0$, and $n        \in \mathds{N}$ such that

(a) $0 < mes(E_1) \leq 1$,

(b) $E_1        \cup        E_2 = [0,1]$,

(c) $\varphi (E_i) = E_i, i=1,2$,

(d) $(\varphi^n)^\prime < (|\lambda| - \varepsilon )^n$ a.e. on $E_2$,

(e)  $(\varphi^n)^\prime < (|\lambda| - \varepsilon )^n$ a.e. on $E_1        \cap [0,1/2]$, and

(f) $(\varphi^{-n})^\prime < (|\lambda| + \varepsilon )^{-n}$ a.e. on $E_1        \cap [1/2,1]$.

\noindent (3)  Let $|\lambda| >1$. Then $ \lambda        \in \sigma_2(T^\prime) \setminus \sigma_1(T)$ if and only if there are Lebesgue measurable disjoint subsets $E_1$ and $E_3$ of $[0,1]$, $\varepsilon > 0$, and $n        \in \mathds{N}$ such that

(a) $0 < mes(E_1) \leq 1$,

(b) $E_1        \cup        E_3 = [0,1]$,

(c) $\varphi (E_i) = E_i, i=1,2$,

(d) $(\varphi^n)^\prime > (|\lambda| + \varepsilon )^n$ a.e. on $E_3$,

(e)  $(\varphi^n)^\prime < (|\lambda| - \varepsilon )^n$ a.e. on $E_1        \cap [0,1/2]$, and

(f) $(\varphi^{-n})^\prime < (|\lambda| + \varepsilon )^{-n}$ a.e. on $E_1        \cap [1/2,1]$.

\noindent (4) Let $|\lambda| = 1$. Then $ \lambda        \in \sigma_2(T^\prime) \setminus \sigma_1(T)$ if and only if $mes(E_1) = 1$, i.e. there are $\varepsilon > 0$, and $n        \in \mathds{N}$ such that

(a)  $(\varphi^n)^\prime < (|\lambda| - \varepsilon )^n$ a.e. on $ [0,1/2]$, and

(b)  $(\varphi^{-n})^\prime < (|\lambda| +  \varepsilon )^{-n}$ a.e. on $ [1/2,1]$.

\end{lemma}

\begin{proof} (1) Assume to the contrary that the set $\sigma_2(T) \setminus \sigma_1(T)$ is not empty. This set is open in $\mathds{C}$ and consists of eigenvalues of $T$. Thus there are a $\lambda        \in \mathds{C}$ and $f        \in Lip_1(0,1)$, $f        \neq        0$, such that $|\lambda|        \neq        1$ and $Tf =  \lambda f$; but that obviously contradicts to $T$ being an isometry of $C[0,1]$.

\noindent (2) $\lambda        \in \sigma_2(T^\prime) \setminus \sigma_1(T)$ if and only if $\lambda        \in \sigma_2(\hat{T}^\prime) \setminus \sigma_1(\hat{T})$.
Recall that $\hat{T}$ is an invertible weighted composition operator on $C(Q)$ defined as
$$ \hat{T}f = \widehat{\varphi^\prime} (f       \circ \hat{\varphi }), f        \in C(Q).$$
Then (see~\cite{Ki3}) the statement that $\lambda        \in \sigma_2(\hat{T}^\prime) \setminus \sigma_1(\hat{T})$ is equivalent to the following. There are closed disjoint subsets $F_i, i=1,2,3$ of $Q$ such that
\begin{itemize}
  \item $F_1$ is clopen and the sets $\hat{\varphi }^n(F_1), n        \in \mathds{Z}$ are pairwise disjoint.
  \item $\hat{\varphi }(F_i) = F_i, i=2,3$.
  \item $Q = F_2        \cup        F_3        \cup        \bigcup \limits_{n=-\infty}^\infty \hat{\varphi }^n(F_1)$.
  \item $\sigma (\hat{T},C(F_2))        \subset \{\xi        \in \mathds{C} : |\xi| < |\lambda|\}$.
  \item $\sigma (\hat{T},C(F_3))        \subset \{\xi        \in \mathds{C} : |\xi| > |\lambda|\}$.
\end{itemize}
Next notice that at least one of the sets $F_2$ or $F_3$ must be nowhere dense in $Q$. Indeed, otherwise let $E_2$ and $E_3$ be the corresponding subsets of positive Lebesgue measure in $[0,1]$. Let $L_i = \{f        \in Lip_1(0,1) : supp \widehat{f^\prime}        \subseteq F_i, i=2,3 \}$. Then $L_i$ is a closed commutative $T$-invariant subalgebra of $Lip_1(0,1)$ and $\sigma (T|L_i) = \sigma (\hat{T},C(F_i)), i=2,3$ in contradiction with the Kamowitz - Scheinberg theorem.

Assume that $|\lambda| \leq 1$ then applying again the Kamowitz - Scheinberg theorem we see that $F_2$ must be nowhere dense in $Q$, while if $|\lambda| \geq 1$ then $ int F_3 =\emptyset$.
\end{proof}

The following example illustrates the statement of Lemma~\ref{l4}

\begin{example} \label{e3} We define the sequences $\{a_n\}, \{b_n\}, n        \in \mathds{Z}$ of numbers in $(0,1)$ in the following way.

$a_n = \begin{cases} \frac{1}{2^{n+1}} &\mbox{if } n \geq 0 \\
1-\frac{1}{2^{|n|+1}} & \mbox{if } n < 0. \end{cases} $,

$b_n = \begin{cases} \frac{1}{2^{n+2}} + \frac{1}{4^{n+2}} &\mbox{if } n \geq 0 \\
1-\frac{1}{2^{|n|+1}} - \frac{1}{4^{|n|+1}} & \mbox{if } n < 0. \end{cases} $.

\noindent We define the homeomorphism $\varphi$ of $[0,1]$ onto itself as follows.
$$\varphi (a_n) = a_{n+1}, \; \varphi (b_n) = b_{n+1}, n        \in \mathds{Z}, $$
\centerline{ $\varphi$ is linear on intervals $[b_n, a_n]$ and $[a_{n-1}, b_n]$, $ n        \in \mathds{Z}$,}
\centerline{$\varphi (0) = 0$ and $\varphi (1) =1$.}
Clearly $\varphi (x) < x, x        \in (0,1)$.

It is immediate that the operator $T$, $Tf = f        \circ \varphi$ is an automorphism of $Lip_1(0,1)$ and that
$$ \sigma (T) = \sigma_2(T^\prime) = \sigma_3(T) = \sigma_4(T) = \sigma_5(T) = \{\lambda        \in \mathds{C} : 1/4 \leq |\lambda| \leq 4 \}. $$
while
$$ \sigma_1(T) = \sigma_2(T) = (1/4)\Gamma        \cup        (1/2)\Gamma        \cup        2\Gamma        \cup        4\Gamma.$$

Let $\lambda        \in \sigma_2(T^\prime) \setminus \sigma_1(T)$ and consider three possibilities.

(a) $1/4 < |\lambda| < 1/2$. Then the sets $A = \bigcup \limits_{-\infty}^\infty [a_{n-1}, b_n)$ and $B = \bigcup \limits_{-\infty}^\infty [a_n, b_n)$ correspond to the sets $E_1$ and $E_2$ from part (2) of the statement of Lemma~\ref{l4}, respectively.

(b) $2 < |\lambda| < 4$. Then the sets $A$ and $B$ correspond to the sets $E_3$ and $E_1$ from part (3) of the statement of Lemma~\ref{l4}, respectively.

(c) $1/2 < |\lambda| < 2$. Then $E_1 = [0,1]$.

\end{example}

Next we will extend the statement of Lemma~\ref{l4} on \textit{weighted} automorphisms of $Lip_1(0,1)$. The case when $\varphi (x) < x$ on $(0,1)$ and $|w(0)| \leq |w(1)|$ does not require any substantial additional effort.

\begin{lemma} \label{l5} Let $\varphi$ be a homeomorphism of $[0,1]$ onto itself such that $\varphi$ is an invertible element of the algebra $Lip_1(0,1)$ and $\varphi (x) <  x$ on $(0,1)$. Let $w$ be an invertible element of $Lip_1(0,1)$ such that $|w(0)| \leq |w(1)|$. Let $Tf = w(f        \circ \varphi ), f        \in Lip_1(0,1)$. Then

\noindent (1) $\sigma_1(T) = \sigma_2(T)$

\noindent (2)  Let $0 < |\lambda| < |w(0)|$. Then $ \lambda        \in \sigma_2(T^\prime) \setminus \sigma_1(T)$ if and only if there are Lebesgue measurable disjoint subsets $E_1$ and $E_2$ of $[0,1]$, $\varepsilon > 0$, and $n        \in \mathds{N}$ such that

(a) $0 < mes(E_1) \leq 1$,

(b) $E_1        \cup        E_2 = [0,1]$,

(c) $\varphi (E_i) = E_i, i=1,2$,

(d) $|w_n| (\varphi^n)^\prime < (|\lambda| - \varepsilon )^n$ a.e. on $E_2$,

(e)  $|w(0)|^n(\varphi^n)^\prime < (|\lambda| - \varepsilon )^n$ a.e. on $E_1        \cap [0,1/2]$, and

(f) $|w(1)|^{-n}(\varphi^{-n})^\prime < (|\lambda| +  \varepsilon )^{-n}$ a.e. on $E_1        \cap [1/2,1]$.

\noindent (3)  Let $|\lambda| >|w(1)|$. Then $ \lambda        \in \sigma_2(T^\prime) \setminus \sigma_1(T)$ if and only if there are Lebesgue measurable disjoint subsets $E_1$ and $E_3$ of $[0,1]$, $\varepsilon > 0$, and $n        \in \mathds{N}$ such that

(a) $0 < mes(E_1) \leq 1$,

(b) $E_1        \cup        E_3 = [0,1]$,

(c) $\varphi (E_i) = E_i, i=1,2$,

(d) $|w_n|(\varphi^n)^\prime > (|\lambda| + \varepsilon )^n$ a.e. on $E_3$,

(e)  $|w(0)|^n(\varphi^n)^\prime < (|\lambda| - \varepsilon )^n$ a.e. on $E_1        \cap [0,1/2]$, and

(f) $|w(1)|^{-n}(\varphi^{-n})^\prime < (|\lambda| + \varepsilon )^{-n}$ a.e. on $E_1        \cap [1/2,1]$.

\noindent (4) Let $|w(0)| \leq|\lambda| \leq |w(1)|$. Then $ \lambda        \in \sigma_2(T^\prime) \setminus \sigma_1(T)$ if and only if $mes(E_1) = 1$, i.e. there are $\varepsilon > 0$, and $n        \in \mathds{N}$ such that

(a)  $|w(0)|^n(\varphi^n)^\prime < (|\lambda| - \varepsilon )^n$ a.e. on $ [0,1/2]$, and

(b)  $|w(1)|^{-n}(\varphi^{-n})^\prime < (|\lambda| +  \varepsilon )^{-n}$ a.e. on $ [1/2,1]$.

\end{lemma}

\begin{proof} The proof repeats almost verbatim the one of Lemma~\ref{l4}. The only difference worth noticing is that instead of the Kamowitz - Scheinberg theorem we use the following fact. Let $E$ be a closed $\varphi$-invariant subset of $[0,1]$ such that $mes(E) > 0$. Let $L_E = \{f        \in Lip_1(0,1) : supp f        \subseteq E\}$. Then $\sigma (T, L_E)        \supseteq w(0)\Gamma        \cup        w(1)\Gamma$.

To prove the last statement let us first notice that because $\varphi (E) = E$ we have $\{w(0), w(1)\}        \subset \sigma (T, L_E)$. Let $\Psi : Q        \rightarrow [0,1]$ be the surjection corresponding to the natural embedding of $C[0,1]$ into $C(Q)$. Let $\hat{E} = \Psi^{-1}(E)$. The operator $\hat{T}$ acts on $C(\hat{E})$ and $\sigma (T, C(\hat{E}))$ is rotation invariant. Assume that $w(0) \Gamma \nsubseteqq \sigma (T,L_E)$. Then $w(0) \Gamma        \cap \sigma (\hat{T},C(\hat{E})) = \emptyset$. We have to consider three possibilities.

(a) $|w(0)| < 1/\rho ((\hat{T})^{-1}, C(\hat{E}))$. Let $F = \Psi^{-1}(0)$. Then there are $N        \in \mathds{N}$ and $\varepsilon >0$ such that
$|w(0)|^n \widehat{(\varphi^n)^\prime}(q) > (|w(0)| + \varepsilon )^n, n \geq N, q        \in F$. Recalling that $\varphi (x) < x$ on $(0,1)$ we see that there is $a        \in (0,1)$ such that $(\varphi^n)^\prime > (1 + \varepsilon/2)^n, n \geq N$ a.e. on $E        \cap [0,a]$, an obvious contradiction.

(b)  $ \rho (\hat{T}, C(\hat{E})) > |w(0)| > 1/\rho ((\hat{T}^{-1}), C(\hat{E})) $. Out assumption that $w(0) \Gamma \nsubseteqq \sigma (T,L_E)$ implies that $\sigma (\hat{T}, C(\hat{E}))        \cap w(0) \Gamma = \emptyset$. Then $\hat{E}$ is the union of two disjoint $\hat{\varphi }$-invariant clopen sets $\hat{E}_1$ and $\hat{E}_2$ such that $|w(0)| < 1/\rho ((\hat{T})^{-1}, C(\hat{E}_2))$. By considering $E_2 = \Psi^{-1}(\hat{E}_2)$ we come to a contradiction in the same way as in part (a).

(c)$ |w(0)| > \rho (\hat{T}, C(\hat{E}))$. Then $ |w(1)| > \rho (\hat{T}, C(\hat{E}))$. Considering the operator $T^{-1}$ and keeping in mind that
$\varphi^{-1}(x) > x$ on $(0,1)$ we again come to a contradiction like in part (a).

\end{proof}

Let us turn to the case when $\varphi (x) < x, x        \in (0,1)$ and $|w(0)| > |w(1)|$.

\begin{lemma} \label{l6} Let $\varphi$ be a homeomorphism of $[0,1]$ onto itself such that $\varphi$ is an invertible element of the algebra $Lip_1(0,1)$ and $\varphi (x) <  x$ on $(0,1)$. Let $w$ be an invertible element of $Lip_1(0,1)$ such that $|w(0)| > |w(1)|$. Let $Tf = w(f        \circ \varphi ), f        \in Lip_1(0,1)$. Then

\noindent (1) $\lambda        \in \sigma_2(T^\prime) \setminus \sigma_1(T)$ if and only if there are Lebesgue measurable disjoint subsets $E_1$, $E_2$, and $E_3$ of $[0,1]$, $\varepsilon > 0$, and $n        \in \mathds{N}$ such that

(a) $0 < mes(E_1) \leq 1$,

(b) $E_1        \cup        E_2        \cup        E_3 = [0,1]$,

(c) $\varphi (E_i) = E_i, i=1, 2, 3$,

(d) $|w_n| (\varphi^n)^\prime < (|\lambda| - \varepsilon )^n$ a.e. on $E_2$,

(e)$|w_n| (\varphi^n)^\prime > (|\lambda| + \varepsilon )^n$ a.e. on $E_3$,

(f)  $|w(0)|^n(\varphi^n)^\prime < (|\lambda| - \varepsilon )^n$ a.e. on $E_1        \cap [0,1/2]$, and

(g) $|w(1)|^{-n}(\varphi^{-n})^\prime < (|\lambda| +  \varepsilon )^{-n}$ a.e. on $E_1        \cap [1/2,1]$.

\noindent (2) $\lambda        \in \sigma_2(T) \setminus \sigma_1(T)$ if and only if there are Lebesgue measurable disjoint subsets $E_1$, $E_2$, and $E_3$ of $[0,1]$, $\varepsilon > 0$, and $n        \in \mathds{N}$ such that

(a) $0 < mes(E_1) \leq 1$,

(b) $E_1        \cup        E_2        \cup        E_3 = [0,1]$,

(c) $\varphi (E_i) = E_i, i=1, 2, 3$,

(d) $|w_n| (\varphi^n)^\prime < (|\lambda| - \varepsilon )^n$ a.e. on $E_2$,

(e)$|w_n| (\varphi^n)^\prime > (|\lambda| + \varepsilon )^n$ a.e. on $E_3$,

(f)  $|w(0)|^n(\varphi^n)^\prime > (|\lambda| + \varepsilon )^n$ a.e. on $E_1        \cap [0,1/2]$, and

(g) $|w(1)|^{-n}(\varphi^{-n})^\prime > (|\lambda| -  \varepsilon )^{-n}$ a.e. on $E_1        \cap [1/2,1]$.

\end{lemma}

The next example shows that the statement of Lemma~\ref{l6} in general cannot be improved.

\begin{example} \label{e4} Let $\{a_n\}, \{b_n\}$, $\{c_n\}$, and $\{d_n\}$, $n        \in \mathds{Z}$ be sequences of real numbers from $(0,1)$ defined as follows
\begin{itemize}
  \item $a_n = \begin{cases} \frac{1}{n+2} &\mbox{if } n \geq 0 \\
1-\frac{1}{|n|+2} & \mbox{if } n < 0. \end{cases} $,
  \item $b_n = \begin{cases} \frac{a_n +a_{n+1}}{2} &\mbox{if } n \geq 0 \\
a_n -\frac{1}{2^{|n|+3}} & \mbox{if } n < 0. \end{cases} $,
  \item $c_n = \begin{cases} a_{n+1} + \frac{1}{2^{n+3}} + (3/4)^{n+3} &\mbox{if } n \geq 0 \\
b_n - (2/3)^{|n|+3} & \mbox{if } n < 0. \end{cases} $,
  \item \item $d_n = \begin{cases} a_{n+1} + \frac{1}{2^{n+3}}  &\mbox{if } n \geq 0 \\
\frac{a_n +a_{n+1}}{2} & \mbox{if } n < 0. \end{cases} $,
\end{itemize}

We define the homeomorphism $\varphi$ of $[0,1]$ onto itself in the following way.
$$\varphi (\Theta_n) = \Theta_{n+1}, n        \in \mathds{Z}, \Theta        \in \{a,b,c,d\}, \varphi (0) = 0, \varphi (1) = 1,$$
and $\varphi$ is linear on the intervals $(a_n, b_n), (b_n, c_n), (c_n, d_n)$, and $(d_n, a_{n+1})$. It is obvious that $\varphi$ is an invertible element of $Lip_1(0,1)$.

Let $w$ be an invertible element of $Lip_1(0,1)$ such that $w(0) = 2$ and $w(1) = 1$. Let $Tf = w(f        \circ \varphi )$. Then it is not difficult to see that
\begin{enumerate}
  \item $\sigma (T) = \sigma_3(T) = \sigma_4(T) = \sigma_5(T) = \{\lambda        \in \mathds{C} : 1 \leq |\lambda| \leq 2\}$.
  \item $\sigma_2(T) = \{\lambda        \in \mathds{C} : 3/2 \leq |\lambda| \leq 2\}$.
  \item $\sigma_2(T^\prime) = 2\Gamma        \cup        \{\lambda        \in \mathds{C} : 1 \leq |\lambda| \leq 3/2\}$.
  \item $\sigma_1(T) = \Gamma        \cup        (3/2)\Gamma        \cup        2\Gamma$.
\end{enumerate}

Moreover, if $\lambda        \in \sigma_2(T) \setminus \sigma_1(T)$ then in notations of Lemma~\ref{l6} we have

\noindent $E_1 = \bigcup \limits_{n=-\infty}^\infty [b_n, c_n)$, $E_2 = \bigcup \limits_{n=-\infty}^\infty [d_n, a_{n+1})$, $E_3 = \bigcup \limits_{-\infty}^\infty (a_n, b_n]$.

On the other hand if $\lambda        \in \sigma_2(T^\prime) \setminus \sigma_1(T)$ then

\noindent $E_1 = \bigcup \limits_{n=-\infty}^\infty [c_n, d_n)$, $E_2 = \bigcup \limits_{n=-\infty}^\infty [a_{n+1}, d_n)$, $E_3 = \bigcup \limits_{-\infty}^\infty [a_n, c_n)$.
\end{example}

We are ready to consider general weighted automorphisms of $Lip_1(0,1)$. First let us notice that the statement of Lemma~\ref{l2} becomes in general false if we drop the assumption that $\varphi (x) < x$ on $(0,1$.

\begin{example} \label{e2} Let $\varphi$ be a homeomorphism of $[0,1]$ onto itself such that
\begin{enumerate}
  \item $\varphi (x) < x$ on $(0,1/2)        \cup        (1/2,1)$.
  \item $\varphi (1/2) = 1/2$.
  \item $\varphi$ and $\varphi^{-1}$ are continuously differentiable on $[0,1/2)        \cup        (1/2,1]$.
  \item $\lim \limits_{x \to 1/2 -} \varphi^\prime(x) = A$, $\lim \limits_{x \to 1/2 +} \varphi^\prime(x) = B$, and $A > B > 0$.
\end{enumerate}
Let $w        \in C^{(1)}[0,1]$ be such that $w(1/2) =1$, $w(0)\varphi^\prime(0) = A $, and $w(1)\varphi^\prime(1) = B $. Let $T$ be the operator on $Lip_1(0,1)$ defined as $Tf = w(f       \circ \varphi )$. Then $\sigma_i(T), i=1, \ldots , 5$ is the union of two circles with the radii $A$ and $B$.
\end{example}

Nevertheless, the previous lemmas provide the following result.

\begin{theorem} \label{t2}. Let $\varphi$ be a homeomorphism of $[0,1]$ onto itself. Assume that $\varphi, \varphi^{-1}        \in Lip_1(0,1)$. Let $w$ be an invertible element of $Lip_1(0,1)$ and let $T: Lip_1(0,1)        \rightarrow Lip_1(0,1)$ be the weighted composition operator defined as
$$Tf =w(f       \circ \varphi ), f        \in Lip_1(0,1).$$
Let $A$ be the weighted composition operator on $L^\infty(0,1)$ defined as
$$Ag =w \varphi^\prime (g       \circ \varphi ), g        \in L^\infty(0,1).$$
Let $F$ be the set of all $\varphi$-periodic points in $[0,1]$ and $I$ the (at most countable) set of points isolated in $F$.
Then $(1)$ $\sigma_i(T) = \sigma_i(A), i=1, \ldots , 5$.

\noindent $(2a)$ If $\varphi (0) = 0$ and $\varphi (1) =1$ then $\sigma (T) \setminus \sigma (A)        \subseteq \{w(t) : t        \in I\}$.

\noindent $(2b)$  If $\varphi (0) = 1$ and $\varphi (1) =0$, and $a$ is the only fixed point of $\varphi$ then $\sigma (T) \setminus \sigma (A)        \subseteq \{w(a)\}        \cup        \{\pm\sqrt{w(t)w(\varphi (t))} : t        \in I \setminus \{a\}\}$.

\noindent $(3)$ The eigenspaces corresponding to the points in $\sigma (T) \setminus \sigma (A)$ are one dimensional.

\noindent $(4)$ If the Lebesgue measure of  $F$ is 0 then the sets $\sigma_i(T), i=1, \ldots , 5$ are rotation invariant.

\end{theorem}

The previous results can be extended to weighted automorphisms of algebras $Lip^n_1(0,1)$.

\begin{theorem} \label{t3}. Let $\varphi$ be a homeomorphism of $[0,1]$ onto itself. Assume that $\varphi, \varphi^{-1}        \in Lip_1^n(0,1)$. Let $w$ be an invertible element of $Lip_1^n(0,1)$ and let $T: Lip_1^n(0,1)        \rightarrow Lip_1^n(0,1)$ be the weighted composition operator defined as
$$Tf =w(f       \circ \varphi ), f        \in Lip_1^n(0,1).$$
Let $A$ be the weighted composition operator on $L^\infty(0,1)$ defined as
$$Ag =w \varphi^{(n+1)} (g       \circ \varphi ), g        \in L^\infty(0,1).$$
Let $F$ be the set of all $\varphi$-periodic points in $[0,1]$ and $I$ the (at most countable) set of points isolated in $F$.
Then $(1)$ $\sigma_i(T) = \sigma_i(A), i=1, \ldots , 5$.

\noindent $(2a)$ If $\varphi (0) = 0$ and $\varphi (1) =1$ then $\sigma (T) \setminus \sigma (A)        \subseteq \{w(t) : t        \in I\}$.

\noindent $(2b)$  If $\varphi (0) = 1$ and $\varphi (1) =0$, and $a$ is the only fixed point of $\varphi$ then $\sigma (T) \setminus \sigma (A)        \subseteq \{w(a)\}        \cup        \{\pm\sqrt{w(t)w(\varphi (t))} : t        \in I \setminus \{a\}\}$.

\noindent $(3)$ The eigenspaces corresponding to the points in $\sigma (T) \setminus \sigma (A)$ are $n$-dimensional.

\noindent $(4)$ If the Lebesgue measure of  $F$ is 0 then the sets $\sigma_i(T), i=1, \ldots , 5$ are rotation invariant.

\end{theorem}

\section{Invertible weighted composition operators on Sobolev spaces $W^{n,X}$.}

Let $X$ be a Banach function space on $(0,1)$, i.e. elements of $X$ are (classes of) Lebesgue measurable functions on $(0,1)$ and if $f       \in X$ and $g$ is a Lebesgue measurable function on $(0,1)$ such that a.e. $|g| \leq |f|$ then $g       \in X$ and $\|g\| \leq \|f\|$.

Assume additionally that $L^\infty(0,1)       \subseteq X       \subseteq L^1(0,1)$. Then we can consider the Sobolev space $W^{1,X}$ defined as follows,
\[
  W^{1,X} = \{f       \in C[0,1] :       \exists       g       \in X \; \text{such that} \; f(x) = f(0) + \int \limits_0^x g(t)dt \}
\]
and endowed with the norm $\|f\| = |f(0)| + \|g\|_X$. If $f       \in W^{1,X}$ then $f$ is a.e. differentiable and its derivative a.e. coincides with $g$; therefore we will write $f^\prime$ instead of $g$.

Similarly, for $n       \in \mathds{N}$ we define the space $W^{n,X}$ as
\[
  W^{n,X} = \{f       \in C^{(n-1)}[0,1] :       \exists       g       \in X \; \text{such that} \; f^{(n-1)}(x) = f^{(n-1)}(0) + \int \limits_0^x g(t)dt \}
\]
It is routine to check that endowed with the norm
\[
  \|f\| = \sum \limits_{i=0}^{n-1} f^{(i)}(0) + \|f^{(n)}\|_X
\]
$W^{n,X}$ is a Banach algebra.

Notice also that when $X = L^p(0,1)$, $1 \leq p \leq \infty$ the space $W^{n,X}$ coincides with the classic Sobolev space $W^{n,p}$. In particular, $W^{n, \infty} = Lip_1^n(0,1)$.

Now we have to make an additional assumption that $X$ is an \textit{interpolation} space (see e.g~\cite[Definition 4.2, page 20]{KPS}) between $L^\infty (0,1)$ and $L^1(0,1)$. In particular, $X$ is a rearrangement-invariant space~\cite{BS}.

\begin{lemma} \label{l7} Let $X$ be an interpolation space between $L^\infty(0,1)$ and $L^1(0,1)$. Let $w$ be an invertible element of the algebra $W^{n,X}$ and $\varphi$ be a homeomorphism of $[0,1]$ onto itself such that $\varphi, \varphi^{-1}        \in W^{n,\infty}$. Let
$$Tf = w(f       \circ \varphi ), f        \in W^{n,X}.$$
Then $T$ is a bounded invertible operator on $W^{n,X}$.

\end{lemma}

\begin{proof}

It is enough to prove that the composition operator $f        \rightarrow f        \circ \varphi$ is bounded on $W^{n,X}$. Let us illustrate the proof  in case $n=2$.
$$(f       \circ \varphi )^\prime = (f^\prime        \circ \varphi )\varphi^\prime        \in C[0,1],$$
$$(f       \circ \varphi )^{\prime \prime} = (f^{\prime \prime}        \circ \varphi )(\varphi^\prime)^2 + (f^\prime        \circ \varphi )\varphi^{\prime \prime}.    $$
It remains to notice that

(a) the composition operator $x        \rightarrow x        \circ \varphi$ is bounded on $X$ because it is clearly bounded on both $L^1[0,1]$ and $L^\infty[0,1]$ and $X$ is an interpolation space,

(b) being a Banach function space $X$ is invariant under multiplication by functions from $L^\infty[0,1]$.

\noindent In general case the proof goes along the same lines with the use of the Fa\`{a} di Bruno's formula for the $n^{th}$ derivative of a composition (see e.g.~\cite{Jo}).
\end{proof}

The next lemma might be known but I was not able to find it in the literature.

\begin{lemma} \label{l8} Let $X$ be an interpolation space between $L^1(0,1)$ and $L^\infty(0,1)$. Then the Volterra operator
$$ (Vx)(t) = \int \limits_0^t x(s) ds, x        \in X, t        \in [0,1], $$
is compact on $X$.

\end{lemma}

\begin{proof} The condition $L^\infty(0,1)        \subseteq X        \subseteq L^1(0,1)$ guarantees that $V$ is a bounded linear operator on $X$. If we can find a sequence of finite dimensional operators $V_n$ such that $V_n$ converges to $V$ in operator norm both on $L^\infty$ and $L^1$ then in virtue of $X$ being an interpolation space $V_n$ will converge to $V$ in operator norm on $X$ (see~\cite[Lemma 4.3, page 20]{KPS}) and we will conclude that $V$ is compact on $X$. To this end let us recall that
$$(Vx)(t) = \int_0^1 K(s,t) x(s) ds, x        \in X, t        \in [0,1]$$
where $K$ is the characteristic function $\chi_\Delta$ of the triangle $\Delta$ with the vertices $(0,0), (0,1)$, and $(1,1)$. Let $Q_1^0$ be the square of generation 0 with the vertices $(0, 1/2), (0, 1), (1/2, 1)$ and $(1/2, 1/2)$. The closure of the difference $\Delta \setminus Q_1^0$ consists of two triangles with disjoint interiors. We will inscribe into these triangles the squares of generation 1 : $Q_1^1$ and $Q_2^1$ similarly to how we inscribed $Q_1^0$ into $\Delta$. By continuing this way we will construct an infinite sequence of squares with pairwise disjoint interiors which will include $2^n$ squares of generation $n$ each with the side length of $\frac{1}{2^{n+1}}$.  Let
$$K_n = \sum \limits_{i=0}^n \sum \limits_{j=1}^{2^n} \chi_{Q_j^i} . $$
The integral operator $V_n$ with the kernel $K_n$ is a finite dimensional operator. Let $W_n = V - V_n$. Then $W_n$ is an integral operator on $L^\infty [0,1]$ and its kernel is the characteristic function of the set $R_n = cl(\Delta \setminus \bigcup \limits_{i=0}^n \bigcup \limits_{j=1}^{2^n} Q_j^i)$. For every $t        \in [0,1]$ the set $R_n        \cap \{(s,t): 0 \leq s \leq 1\}$ is an interval $\{(s,t): a(t) \leq s \leq b(t)\}$ where $0 \leq b(t) - a(t) \leq \frac{1}{2^{n+1}}$. Let $x        \in L^\infty$. Then for any $t        \in [0,1]$
$$ |(W_n x)(t)|  \leq \int \limits_{a(t)}^{b(t)} |x(s)| ds \leq \frac{1}{2^{n+1}} \|x\|_\infty . $$
 Therefore $\|W_n \| \leq \frac{1}{2^{n+1}} $.

Now consider the operator $W_n$ on $L^1[0,1]$. Then the adjoint $W_n^\prime$ acts on $L^\infty [0,1]$. It is an integral operator with the kernel
$S(s,t) = K(t,s) - K_n(t,s)$. Then the same reasoning as above shows that $\|W_n^\prime\| \leq \frac{1}{2^{n+1}}$.
\end{proof}

\begin{lemma} \label{l9} Assume that $X$, $w$, and $\varphi$ satisfy conditions of Lemma~\ref{l7}. Let $T: Tf = w(f       \circ \varphi ), f        \in W^{n,X}$ be the corresponding weighted composition operator. Assume additionally that $\varphi (0) = 0$. Consider the subspace $W_0$ of $W^{n,X}$ defined as
$$W_0 = \{f        \in W^{n,X}: f^{(j)}(0) = 0, j= 0,1, \ldots , n-1 \}.$$
Then $TW_0        \subseteq W_0$ and the restriction of $T$ on $W_0$ is similar to the operator $A +K$ on $X$, where
$$Ax = w (\varphi^\prime)^n ( x       \circ \varphi ), x        \in X \eqno{(4.0)}$$
and $K$ is a compact operator on $X$.
\end{lemma}

\begin{proof} Like in the proof of Lemma~\ref{l7} we will look at the case $n=2$. Then
$$(w(f       \circ \varphi ))^\prime = w^\prime (f       \circ \varphi ) + w \varphi^\prime (f^\prime        \circ \varphi )$$
and
$$(w(f        \circ \varphi ))^{\prime \prime} = w^{\prime \prime}(f       \circ \varphi ) + 2w^\prime \varphi^\prime (f^\prime        \circ \varphi ) + w \varphi^{\prime \prime} (f^\prime        \circ \varphi ) + w (\varphi^\prime)^2 (f^{\prime \prime}        \circ \varphi ) . $$
It follows that $TW_0        \subseteq W_0$. Next notice that the norm on $W^{2,X}$ is equivalent to the norm
$$\|f\| = |f(0)| + |f^\prime(0)| + \|f^{\prime \prime}\|_X.$$
 Therefore the operator $Jf = f^{\prime \prime}$ is an isomorphism between $W_0$ and $X$, and for almost all $t        \in [0,1]$ we have
$$ (JTJ^{-1}x)(t) = w(t)[\varphi^\prime (t)]^2 x(\varphi (t)) + w(t) \varphi^{\prime \prime}(t) \int \limits_0^{\varphi (t)} x(s) ds $$
$$ + \;  2w^\prime(t) \varphi^\prime(t)\int \limits_0^{\varphi (t)} x(s) ds + w^{\prime \prime}(t) \int \limits_0^{\varphi (t)} \Big{(} \int \limits_0^s x(u) du \Big{)} ds .      \eqno{(4.1)} $$
The conclusion of the lemma for $n=2$ follows now from $(4.1)$ and Lemma~\ref{l8}.

For arbitrary $n$ we can use the standard formula for the $n^{th}$ derivative of a product and the Fa\`{a} di Bruno's formula.
\end{proof}

\begin{corollary} \label{c4} $\sigma_i(T) = \sigma_i(A), i= 1,2,3,4$, and $\sigma_2(T^\prime)= \sigma_2(A^\prime)$.
\end{corollary}

Recall that a Banach lattice $X$ is said to have the Fatou property if for any positive $x       \in X$ and for any net $\{x_\alpha \}       \subset X$ such that $x_\alpha \geq 0$ and $x_\alpha \nearrow x$ we have $\sup \limits_\alpha \|x_\alpha \| = \|x\|$.

If a Banach function space $X$, $L^\infty (0,1)       \subseteq X       \subseteq L^1 (0,1)$, has the Fatou property then $X$ is an interpolation space (see~\cite{KPS}).

\begin{theorem} \label{t6} Let $X$ have the Fatou property and let $T$ be an invertible weighted automorphism of $W^{n,X}$, $\varphi (0) =0$, and $mes(F) = 0$ where $F = \{t        \in [0,1] : \varphi (t) = t\}$. Let $A$ be the operator on $X$ defined by (4.0). Then

\noindent (1) The sets $\sigma_i(T), i=1,2,3,4$, and the set $\sigma_2(T^\prime)$ are rotation invariant.

\noindent (2)$\sigma (A)        \subseteq \sigma (T)$.

\noindent (3) The set $\tau = \sigma (T) \setminus \sigma (A)$ consists of isolated eigenvalues of $T$ of finite multiplicity, and therefore is at most countable.

\noindent (4) $\sigma_{a.p}(T) = \sigma_2(A)     \cup      \tau$ and $\sigma_{a.p}(T^\prime) = \sigma_2(A^\prime)        \cup        \tau$.

\end{theorem}

\begin{proof} (1). By Corollary~\ref{c4} in order to prove (1) it is enough to prove that the sets $\sigma_i(A), i=1,2,3,4$, are rotation invariant.
Let $\alpha        \in \mathds{C}$, $|\alpha| = 1$. Let $Q$ be the Gelfand compact of the algebra $L^\infty(0,1)$. Then $Q$ is a   hyperstonian compact space. Let $\tilde{\varphi }$ be the homeomorphism of $Q$ corresponding to the invertible disjointness preserving operator $A$ (see~\cite{AAK}). The condition $mes(F) =0$ and the Frolik's theorem~\cite{Wa} guarantee that $\tilde{\varphi }$ has no periodic points in $Q$. By Theorem B.1 in~\cite[p. 147]{AAK} for any positive $\varepsilon$ there is $g        \in C(Q)$ such that
$|g|        \equiv        1$ on $Q$ and $\|g        \circ \varphi - \alpha g\| \leq \varepsilon$. Let $G$ be the operator on $X$ of multiplication on $g$. Then for any $x       \in X$
$$\|(G^{-1}AG - \alpha A)x\| = \|g^{-1}(g       \circ \varphi )Ax - \alpha Ax\| \leq \varepsilon \|A|\|\|x\|. \eqno{(4.2)}$$
It follows immediately from (4.2) that the sets $\sigma_i(A), i=1,2,3,4$ are rotation invariant.

\noindent (2) Let $\lambda       \in \sigma (A) \setminus \sigma (T)$. Assume first that $\lambda       \in \sigma_{a.p.}(A)$. By Corollary~\ref{c4} $\lambda \not       \in \sigma_2(A)$ and therefore $\lambda$ is an eigenvalue of $A$. Let $f       \in X$ be a corresponding eigenfunction. Because $mes(F) =0$ we can find an interval $(a,b)       \subseteq (0,1)$ such that $[a,b]       \cap F = \{a,b\}$ and $f \not       \equiv       0$ on $(a,b)$. Let $E$ be a subset of $(a,b)$ such that $mes(E) > 0$, $|f| > c > 0$ on $E$, and the sets $\varphi^i(E), i       \in \mathds{Z}$ are pairwise disjoint. Let $E_1$ be a measurable subset of $E$ of positive measure and
$G=\bigcup \limits_{i       \in \mathds{Z}} \varphi^i(E_1)$. Then clearly $\chi_G f       \in \ker{(\lambda I - A)}$. Thus we can construct a sequence $f_n$ of pairwise disjoint elements of $X$ such that $\|f_n\| = 1$ and $f_n       \in \ker{(\lambda I - A)}$, in contradiction with our assumption that $\lambda \not       \in \sigma_2(A)$.

Next assume that $\lambda       \in \sigma_r(A)$. Notice that by Corollary~\ref{c4} $\lambda \not       \in \sigma_2(A^\prime)$. Because $X$ is a Banach function space the set $X^\prime_n$ of order continuous bounded functionals on $X$ is total on $X$ and moreover $X^\prime$ is the direct sum of disjoint bands $X^\prime_n$ and $X^\prime_s$ - the band of bounded singular functionals on $X$ (see e.g.~\cite{KA}). The band $X_n^\prime$ is a Banach function space on $[0,1$. Next, because $A^\prime$ and $(A^\prime)^{-1}$ are order continuous operators on $X^\prime$ (indeed they are invertible disjointness preserving operators) we have $A^\prime X^\prime_n = X^\prime_n$ and $A^\prime X^\prime_s = X^\prime_s$. Because $(\lambda I - A^\prime)X^\prime = X^\prime$ we have $(\lambda I - A^\prime)X^\prime_n = X^\prime_n$. If the operator $\lambda I - A^\prime$ were invertible on $X_n\prime$ then $\lambda I - A^{\prime \prime}$ would be invertible on $(X_n^\prime)_n^\prime$. But, because $X$ has the Fatou property it is $(o)$-reflexive, i.e. $(X_n^\prime)_n^\prime = X$ and therefore $\lambda I - A$ would be invertible on $X$, a contradiction. We conclude that there is a $\Phi       \in X_n^\prime$ such that $\Phi       \neq       0$ and $A^\prime \Phi = \lambda \Phi$.
Finally notice that
\[
  A^\prime (fx) = (f       \circ \varphi^{-1}) Ax, \; x       \in X_n^\prime , \; f       \in L^\infty (0,1) \eqno{(4.3)}
\]
By using (4.3) and the same kind of reasoning as in the case when $\lambda       \in \sigma_{a.p.}(A)$ we come to a contradiction.

\noindent (3) Let $\lambda     \in \sigma (T) \setminus \sigma (A)$. Because $mes(F)=0$ the set $\sigma (A)$ is rotation invariant (see~\cite[Theorem 13.4, page 110]{AAK}). Therefore there are three possibilities.

(a) $|\lambda| > \rho (A) = \rho_e(T)$. Then it is well known (see e.g~\cite[Theorem 4, page 173]{Mu}
or~\cite[Theorem 4.3.18, page 122]{Da}) that $\lambda$ is an isolated eigenvalue of $T$ of finite multiplicity.

(b) $1/|\lambda| > \rho (A^{-1})$. Then again $\lambda$ is an isolated eigenvalue of $T$ of finite multiplicity because $\rho_e(T^{-1})=\rho (A^{-1})$.

(c) \[
      \frac{1}{\rho (A^{-1})} < |\lambda| < \rho (A)
    \]
    and $\lambda \Gamma     \cap \sigma (A) = \emptyset$. In this case (see~\cite[Theorem 13.1, page 106]{AAK}) $X$ is the direct sum of two disjoint $A$-invariant bands $X_1$ and $X_2$ such that
    \[
      \sigma (A, X_1) = \sigma^1 = \sigma (A)     \cap \{\alpha     \in \mathds{C} : |\alpha| < |\lambda|\}
    \]
    and
    \[
      \sigma (A, X_2) = \sigma^2 = \sigma (A)     \cap \{\alpha     \in \mathds{C} : |\alpha| > |\lambda|\}.
    \]
    Let $E_i = supp(X_i), i=1,2$. Then $E_1, E_2$ are disjoint $\varphi$-invariant subsets of $[0,1]$, $mes(E_i) >0, i=1,2$, and $mes(E_1) + mes(E_2)=1$. Let $W_0$ be the subspace of $W^{n,X}$ introduced in Lemma~\ref{l9}. Then $W_0 = W_0^1     \oplus     W_0^2$ where
    \[
      W_0^i = \{f     \in W_0 : f^{(n)}     \in X_i, i=1,2\} .
    \]
It is easy to see that $TW_0^i     \subseteq W_0^i, i=1,2 $, $\rho_e(T,W_0^1)=\rho (A, X_1)$, and $\rho_e(T^{-1}, W_0^2) = \rho (A^{-1}, X_2)$. The statement that $\lambda$ is an isolated eigenvalue of $T$ of finite multiplicity follows from the fact that $dim(W^{n,X}/W_0) < \infty$.
\noindent (2). It suffices to notice that the proof of part (1) above shows that $\sigma_{a.p.}(A) = \sigma_2(A)$ and $\sigma_{a.p.}(A) = \sigma_2(A)$.

\end{proof}

\begin{remark} \label{r4.1} While most of the usually considered rearrangement-invariant Banach function spaces have the Fatou property it is worth noting that the statement of Theorem~\ref{t6} remains true without assuming it. Indeed, it is enough to prove that $\sigma_{a.p}(A^\prime) = \sigma_2(A^\prime)$. Assume that $\lambda     \in \sigma_{a.p}(A^\prime) \setminus \sigma_2(A^\prime)$.  The proof of Lemma 5 in~\cite{Ki2} shows that there are pairwise disjoint elements $u_i        \in X^\prime, i        \in \mathds{Z}$ such that $u_0$ is not a finite linear combination of atoms, $A^\prime u_i = \lambda u_{i+1}, i        \in \mathds{Z}$, and $\sum \limits_{-\infty}^\infty \|u_i\| < \infty$. That immediately implies a contradiction with
$dim \ker{(\lambda I - A^\prime)} < \infty$.

\end{remark}

We can obtain more detailed information about the essential spectra of invertible weighted compositions on $W^{n,X}$ at the expense putting more stringent conditions on $X$. We start with the following proposition.

\begin{proposition} \label{p1} Let $X$ be a rearrangement-invariant Banach space intermediate between $L^1(0,1)$ and $L^\infty(0,1)$ and such that its Boyd indices (see~\cite[page 165]{BS}) satisfy the condition $0 < \alpha_X = \beta_X < 1$. Let $\varphi$ be an invertible element of $Lip_1(0,1)$ that  maps $[0,1]$ onto itself. Let $p = \frac{1}{\alpha_X}$.  Then the operator $U$,
$$ Ux = (\varphi^\prime)^{\frac{1}p{}} (x        \circ \varphi ), x        \in X$$
is bounded on $X$ and moreover $\sigma (U)        \subseteq \Gamma$.
\end{proposition}

\begin{proof} Let $\varepsilon > 0$. By Boyd's interpolation theorem ( sse e.g.\cite[Theorem 5.16, page 153 and the discussion after Corollary 6.11, page 165]{BS}) for any small enough $\varepsilon$ the space $X$ will be an interpolation space between $L^{p-\varepsilon }(0,1)$ and $L^{p+\varepsilon }(0,1)$. Let $C(\varepsilon )$ be the corresponding interpolation constant~\cite[Lemma 4.3, page 20]{KPS}. Consider an $n        \in \mathds{Z}$ and the action of the operator $U^n$ on $L^{p-\varepsilon }(0,1)$ and $L^{p+\varepsilon }(0,1)$. We have
$$U^n =  ((\varphi^n)^\prime)^{-\frac{\varepsilon }{p(p - \varepsilon )}} ((\varphi^n)^\prime)^\frac{1}{p-\varepsilon } (x       \circ \varphi^n) =  ((\varphi^n)^\prime)^{-\frac{\varepsilon }{p(p - \varepsilon )}} U^n_{-\varepsilon } $$
and
$$U^n =  ((\varphi^n)^\prime)^{\frac{\varepsilon }{p(p + \varepsilon )}} ((\varphi^n)^\prime)^\frac{1}{p-\varepsilon } (x       \circ \varphi^n) =  ((\varphi^n)^\prime)^{\frac{\varepsilon }{p(p + \varepsilon )}} U^n_{+\varepsilon } $$
where $U_{-\varepsilon }$ and $U_{+\varepsilon }$ are invertible isometries of $L^{p-\varepsilon }(0,1)$ and $L^{p+\varepsilon }(0,1)$, respectively. Consequently we get
$$ \|U^n\|_{L(X)}^\frac{1}{|n|} \leq (C(\varepsilon ))^\frac{1}{|n|} \max{(\|\varphi^\prime\|_\infty^{-\frac{\varepsilon }{p(p - \varepsilon )}} , \|\varphi^\prime\|_\infty^{\frac{\varepsilon }{p(p + \varepsilon )}}, \|\frac{1}{\varphi^\prime}\|_\infty^{-\frac{\varepsilon }{p(p - \varepsilon )}}, \|\frac{1}{\varphi^\prime}\|_\infty^{\frac{\varepsilon }{p(p + \varepsilon )}})}, n        \in \mathds{Z}$$
whence $$\max{(\rho (U), \rho (U^{-1}))} \leq \max{(\|\varphi^\prime\|_\infty^{-\frac{\varepsilon }{p(p - \varepsilon )}} , \|\varphi^\prime\|_\infty^{\frac{\varepsilon }{p(p + \varepsilon )}}, \|\frac{1}{\varphi^\prime}\|_\infty^{-\frac{\varepsilon }{p(p - \varepsilon )}}, \|\frac{1}{\varphi^\prime}\|_\infty^{\frac{\varepsilon }{p(p + \varepsilon )}})} .$$
Finally, because $\varepsilon$ is arbitrary small we have $\rho (U) = \rho (U^{-1}) = 1$.
\end{proof}

Proposition~\ref{p1} combined with Theorem 4.5 in~\cite[page 248]{Ki3} provides the following result.

\begin{theorem} \label{t4.1} Let $X$ be a rearrangement-invariant Banach space intermediate between $L^1(0,1)$ and $L^\infty(0,1)$ and such that its Boyd indices satisfy the condition $0 < \alpha_X = \beta_X < 1$. Let $w, \varphi$ be invertible elements of $Lip_1^n(0,1)$ and assume that $\varphi$  maps $[0,1]$ onto itself. Let $p = \frac{1}{\alpha_X}$. Let $\tilde{T}$ be an invertible operator on $Lip_1^n(0,1)$ defined as
\[
  \tilde{T}f = w \big{(} (\varphi^\prime)^{n - n/p} \big{)} f     \circ \varphi , f     \in Lip_1^n(0,1).
\]
Then $\sigma_i(T) = \sigma_i(\tilde{T}), i=1, \ldots , 5.$
\end{theorem}

\begin{remark} \label{r4.2} Conditions of Theorem~\ref{t4.1} are satisfied in particular in the case of classic Sobolev spaces $W^{1,p}$, $1 < p < \infty$. Moreover the conclusion of Theorem~\ref{t4.1} as we know already is true for $W^{1, \infty} = Lip_1(0,1)$ and it can be easily extended to the space $W^{1,1}$.

\end{remark}

\begin{corollary} \label{c4.3} Assume conditions of Theorem~\ref{t4.1}. Let $\tilde{w} = w \big{(} (\varphi^\prime)^{n - n/p} \big{)} $ Then
\[
  \sigma (T) \setminus \sigma (A)     \subseteq \tilde{w}(F).
\]

\end{corollary}

\begin{corollary} \label{c4.2} Assume conditions of Theorem~\ref{t4.1}. Assume additionally that $F=\{0,1\}$. Then $\sigma (A)$ is an annulus or circle centered at $0$ and $\sigma (T) \setminus \sigma (A)     \subseteq \{\tilde{w}(0), \tilde{w}(1)\}$.

\end{corollary}

\section{Some related problems.}

The results of the current paper are based on the following two circumstances:

(a) Compactness of the Volterra operator.

(b) Exceptionally simple dynamics of homeomorphisms of an interval.

Therefore it probably will be not very hard to extend these results to the class of non-invertible operators when $\varphi$ is a homeomorphism of $[0,1]$ such that $\varphi$ belongs to the algebra $C^{(n)}[0,1]$ (or $Lip1^n[0,1]$) but is not invertible in this algebra. The same relates to the case when the weight $w$ is not invertible in the corresponding algebra.

The following problems seem to be more challenging.

\noindent (1) Description of essential spectra of weighted composition operators in $C^{(n)}[0,1]$ or $W^{n,X}[0,1]$ in the case when $\varphi$ is an arbitrary smooth mapping of $[0,1]$ into itself. While the spectrum of weighted compositions in $C(K)$ is described (see e.g.`\cite{Ki1}) and we have some information about the spectrum of non-invertible disjointness preserving operators on Banach lattices (see~\cite{AAK}) , not much is known about \textit{essential spectra} of such operators even in the case of $C(K)$.

\noindent (2) The case of the infinite interval $(-\infty, \infty)$. One of the main difficulties here is that the corresponding Volterra operator is not compact.

\noindent (3) The case of spaces of smooth functions defined in a domain in $\mathds{R}^m, m>1$.

\end{document}